\documentclass[12pt]{article}

\usepackage{amsfonts}
\usepackage{graphicx}
\usepackage{amsmath}
\usepackage{amssymb}
\usepackage[numbers,square]{natbib}
\usepackage{url}
\usepackage{multirow}
\usepackage{color}
\usepackage{amsthm}
\usepackage{bm}

\newtheorem{theorem}{Theorem}
\newtheorem{corollary}[theorem]{Corollary}

\newtheorem{lemma}[theorem]{Lemma}
\newtheorem{proposition}[theorem]{Proposition}

\title{On magic factors in Stein's method for compound Poisson approximation}
\author{Fraser Daly\footnote{Department of Actuarial Mathematics and Statistics and the Maxwell Institute for Mathematical Sciences, Heriot-Watt University, Edinburgh EH14 4AS, UK.  E-mail: f.daly@hw.ac.uk;  Tel: +44 (0)131 451 3212; Fax: +44 (0)131 451 3249}}
\date{\today}

\begin{document}

\maketitle

\noindent{\bf Abstract} One major obstacle in applications of Stein's method for compound Poisson approximation is the availability of so-called magic factors (bounds on the solution of the Stein equation) with favourable dependence on the parameters of the approximating compound Poisson random variable.  In general, the best such bounds have an exponential dependence on these parameters, though in certain situations better bounds are available. In this paper, we extend the region for which well-behaved magic factors are available for compound Poisson approximation in the Kolmogorov metric, allowing useful compound Poisson approximation theorems to be established in some regimes where they were previously unavailable.  To illustrate the advantages offered by these new bounds, we consider applications to runs, reliability systems, Poisson mixtures and sums of independent random variables.
\vspace{12pt}


\vspace{12pt}


\section{Introduction}

In recent years, Stein's method has proved to be a versatile technique for proving explicit compound Poisson approximation results in a wide variety of settings; see \cite{bc01} and references therein for an introduction to these techniques and a discussion of several applications.  One of the difficulties in applying Stein's method for compound Poisson approximation is the availability of good bounds on the solution of the so-called Stein equation in this setting; the availability of such favourable bounds depends upon the parameters of the compound Poisson distribution in question satisfying one of a number of conditions, which we discuss further below.  Our purpose here is to show how one such condition can be generalized and relaxed.

We say that $U\sim\mbox{CP}(\lambda,\bm{\mu})$ has a compound Poisson distribution if $U$ is equal in distribution to $\sum_{j=1}^NX_j$, where $N\sim\mbox{Po}(\lambda)$ has a Poisson distribution with mean $\lambda$, and $X,X_1,X_2,\ldots$ are i$.$i$.$d$.$ with $\mathbb{P}(X=k)=\mu_k$.  We write $\bm{\mu}=(\mu_1,\mu_2,\ldots)$, and let $\lambda_k=\lambda\mu_k$.  We also write
$$
\theta_k=\sum_{j=1}^\infty j(j-1)\cdots(j-k)\lambda_j\,,
$$
for $k=0,1,\ldots$.

Stein's method for compound Poisson approximation, as developed in \cite{bcl92}, begins by finding, for each function $h$ in some given class $\mathcal{H}$, a function $f_h$ solving
\begin{equation}\label{eq:cpstein}
h(x)-\mathbb{E}h(U)=\sum_{j=1}^\infty j\lambda_jf_h(x+j)-xf_h(x)\,,
\end{equation}
for each $x\in\mathbb{Z^+}=\{0,1,2\ldots\}$, where $U\sim\mbox{CP}(\lambda,\bm{\mu})$.   We may then assess the quality of the approximation of a non-negative, integer-valued random variable $W$ by the compound Poisson random variable $U$ by bounding the right-hand side of the following equality:
\begin{equation}\label{eq:stein}
\sup_{h\in\mathcal{H}}\left|\mathbb{E}h(W)-\mathbb{E}h(U)\right|
=\sup_{h\in\mathcal{H}}\left|\mathbb{E}\left[\sum_{j=0}^\infty j\lambda_jf_h(W+j)-Wf_h(W)\right]\right|\,.
\end{equation}
If we choose $\mathcal{H}$ to be the set $\mathcal{H}_{TV}=\left\{I(\cdot\in A):A\subseteq\mathbb{Z}^+\right\}$, the left-hand side of (\ref{eq:stein}) becomes the total variation distance between $W$ and $U$.  In this note, our primary interest is in the Kolmogorov distance, defined by 
$$
d_K(\mathcal{L}(W),\mathcal{L}(U))=\sup_{y\in\mathbb{Z}^+}\left|\mathbb{P}(W\leq y)-\mathbb{P}(U\leq y)\right|\,,
$$ 
which can be obtained from (\ref{eq:stein}) by choosing $\mathcal{H}$ to be the class $\mathcal{H}_K=\left\{I(\cdot\leq y):y\in\mathbb{Z}^+\right\}$.

In bounding the right-hand side of (\ref{eq:stein}), it is essential to have bounds controlling the behaviour of $f_h$.  This is typically achieved (for the Kolmogorov distance) by finding upper bounds on
$$
M_l^{(K)}=M_l^{(K)}(U)=\sup_{h\in\mathcal{H}_K}\sup_{x\in\mathbb{Z}^+}|\Delta^lf_h(x)|\,,
$$ 
for $l=0,1$, where $\Delta$ denotes the forward difference operator, so that $\Delta f(x)=f(x+1)-f(x)$ for any function $f$.  Such bounds are often referred to as Stein factors, or magic factors.

Similarly, when using (\ref{eq:stein}) to bound the total variation distance between $W$ and $U$, upper bounds for $M_0^{(TV)}$ and $M_1^{(TV)}$ are required, where these quantities are defined analogously to the above, but with $\mathcal{H}_K$ replaced by $\mathcal{H}_{TV}$.  Note that $M_l^{(K)}\leq M_l^{(TV)}$ for each $l$ (since $\mathcal{H}_K\subseteq\mathcal{H}_{TV}$), and so Stein factors for total variation distance may also be employed when considering approximation in Kolmogorov distance.

Unfortunately, good Stein factors for compound Poisson approximation are often not readily available.  Barbour et al. \cite[Theorem 4]{bcl92} show that
\begin{equation}\label{eq:genbd}
M_0^{(TV)},M_1^{(TV)}\leq\min\left\{1,\frac{1}{\lambda_1}\right\}e^\lambda\,,
\end{equation}
and that, in general, this dependence on $\lambda$ cannot be improved.  Such bounds are therefore useful only for small $\lambda$.  However, there are certain conditions on the $\lambda_j$ under which better Stein factors are available for the corresponding compound Poisson random variable.  For example, if we assume that
\begin{equation}\label{eq:inc}
j\lambda_j\geq(j+1)\lambda_{j+1}\,,
\end{equation}
for all $j\geq1$, then \cite[Proposition 1.1]{bx00} shows that
\begin{equation}\label{eq:bx00}
M_0^{(K)}\leq\min\left\{1,\sqrt{\frac{2}{e\lambda_1}}\right\}\,,\qquad M_1^{(K)}\leq\min\left\{\frac{1}{2},\frac{1}{\lambda_1+1}\right\}\,,
\end{equation}
which vastly improves the corresponding compound Poisson approximation bounds.  An upper bound on $M_1^{(TV)}$ is also established by \cite[Theorem 5]{bcl92} under the same condition (\ref{eq:inc}).  This bound has dependence on the $\lambda_j$ which is not quite as good as that in (\ref{eq:bx00}), as it includes an undesirable logarithmic term.

Barbour and Utev \cite{bu98} use Fourier techniques to relax the condition (\ref{eq:inc}) somewhat, and establish Stein factors for compound Poisson approximation in Kolmogorov distance of a better order than is generally available.  Unfortunately, their bound on $M_1^{(K)}$ again includes an undesirable logarithmic term, which can only be removed at the cost of a significantly increased constant in the bound.

Their bounds are proved by making the choice of test function $h(x)=t^x-\mathbb{E}t^U$ in (\ref{eq:cpstein}), for $t\in\mathbb{C}$ with $|t|=1$.  With this choice of test function, Barbour and Utev \cite[Theorem 2.1]{bu98} show that the equation (\ref{eq:cpstein}) is solved by
\begin{equation}\label{eq:sol}
f_h(x)=\int_t^1u^{x-1}e^{\lambda[\mu(t)-\mu(u)]}\,du\,,
\end{equation}
for $x\geq1$, where $\mu(t)=\sum_{j=1}^\infty\mu_jt^j$ and the integral is taken along any contour in the unit disc in $\mathbb{C}$ from $t$ to 1.  The solution to the equation (\ref{eq:cpstein}) for any $h\in\mathcal{H}_K$ may then be written in terms of functions of the form (\ref{eq:sol}).  This allows bounds on the $M_l^{(K)}$ to be found using the following result, which is proved as part of Theorem 3.1 of \cite{bu98}:
\begin{theorem}[\cite{bu98}]\label{thm:bu}
Let $f_h$ be given by (\ref{eq:sol}) and assume that $\theta_0<\infty$.  If there exists $\delta>0$ such that
$$
|f_h(x)|\leq\frac{|1-t|}{\delta(1-\mbox{Re}[t])}\quad\mbox{ and }\quad|\Delta f_h(x)|\leq\frac{|1-t|^2}{\delta(1-\mbox{Re}[t])}\,,
$$
for all $t\in\mathbb{C}$ with $|t|=1$, then the corresponding compound Poisson random variable $U\sim\mbox{CP}(\lambda,\bm{\mu})$ satisfies
$$
M_0^{(K)}\leq2\sqrt{\frac{2}{\delta}}\quad\mbox{ and }\quad M_1^{(K)}\leq\frac{1}{2\delta}\left[1+\log^+(\pi\delta)\right]\,,
$$
$\log^+$ denoting the positive part of the natural logarithm.
\end{theorem} 

An alternative condition, different in flavour to (\ref{eq:inc}), is also available, under which Stein factors also exhibit more favourable dependence on the $\lambda_j$ than is generally possible.  Barbour and Xia \cite[Theorem 2.5]{bx99} prove that if our compound Poisson random variable satisfies
\begin{equation}\label{eq:bxcond}
\theta_0-2\theta_1>0\,,
\end{equation}  
then 
\begin{equation}\label{eq:bx99}
M_0^{(TV)}\leq\frac{\sqrt{\theta_0}}{\theta_0-2\theta_1}\,,\quad\mbox{ and }\quad M_1^{(TV)}\leq\frac{1}{\theta_0-2\theta_1}\,.
\end{equation}

Our purpose in this note is to show how the Fourier techniques embodied in Theorem \ref{thm:bu} can be used to generalize and relax the condition (\ref{eq:bxcond}) when finding reasonable Stein factors for compound Poisson approximation in Kolmogorov distance.  This will extend the applicability of various compound Poisson approximation results in the literature, allowing approximation theorems with a reasonable error bound to be established for previously inaccessible parameter values.  We illustrate this in Section \ref{sec:2} with applications to runs and reliability systems, where good compound Poisson estimates are established for a larger range of parameter values than was previously available. 

Unfortunately, as we are taking advantage of Theorem \ref{thm:bu}, our bounds on $M_1^{(K)}$ will include the undesirable logarithmic factor.  The dependence on the $\lambda_j$ will, of course, still be superior to the exponential bound available in the general case, and allow us to approach approximation problems for which the conditions (\ref{eq:inc}) or (\ref{eq:bxcond}) do not hold.


Our generalization of (\ref{eq:bxcond}), by modifying the inequality to include $\theta_j$ for $j>1$, is presented in Section \ref{sec:1}.  Applications to runs and reliability systems are then given in Section \ref{sec:2}.  Section \ref{sec:gen2} will present a further relaxation of inequality (\ref{eq:bxcond}), allowing good Stein factors to be established in some cases where $\theta_0<2\theta_1$.  Some examples are given to illustrate this result.

\section{A generalization of the Barbour--Xia condition}\label{sec:1}

We use this section to prove our main theorem, the following generalization of the condition (\ref{eq:bxcond}):
\begin{theorem}\label{thm:main}
Let $k\in\{1,2,\ldots\}$ and $U\sim\mbox{CP}(\lambda,\bm{\mu})$.  Define $g_k:(-\pi,\pi]\times[0,1]\mapsto\mathbb{R}$ by
$$
g_k(\phi,p)=\frac{1}{\cos\phi-1}\sum_{j=1}^k\frac{\mbox{Re}[(e^{i\phi}-1)^j]}{j!}\frac{\left(1-(1-p)^j\right)}{p}\theta_{j-1}-\frac{2^{k}}{k!}\theta_k\,.
$$
Let 
$$
\delta_k=\inf_{\phi,p}g_k(\phi,p)\,.
$$  
Assume that $\theta_0<\infty$ and $\delta_k>0$.  Then $U$ satisfies
$$
M_0^{(K)}\leq2\sqrt{\frac{2}{\delta_k}}\quad\mbox{ and }\quad M_1^{(K)}\leq\frac{1}{2\delta_k}\left[1+\log^+(\pi\delta_k)\right]\,.
$$
\end{theorem}

The proof of Theorem \ref{thm:main} is given in Section \ref{sec:proof} below.  Applications will follow in Section \ref{sec:2}.

Choosing $k=1$ in Theorem \ref{thm:main}, the condition $\delta_1>0$ is easily shown to be equivalent to (\ref{eq:bxcond}).  The condition $\delta_2>0$ is actually stronger than (\ref{eq:bxcond}).  However, for $k\geq3$ we obtain a condition which will, in certain situations, be weaker than (\ref{eq:bxcond}).  In Section \ref{sec:2}, our applications will mainly employ the case $k=3$.  In this case, we have the following result:
\begin{corollary}\label{cor:3}
Let $U\sim\mbox{CP}(\lambda,\bm{\mu})$.  Assume that $\theta_0<\infty$ and $\theta_2<2\theta_1$.  If
$$
\delta=\theta_0-2\theta_1+2\theta_2-\frac{4}{3}\theta_3>0\,,
$$
then $U$ satisfies
$$
M_0^{(K)}\leq2\sqrt{\frac{2}{\delta}}\quad\mbox{ and }\quad M_1^{(K)}\leq\frac{1}{2\delta}\left[1+\log^+(\pi\delta)\right]\,.
$$
\end{corollary}
\begin{proof}
Choosing $k=3$ we have
$$
g_3(\phi,p)=\theta_0+(\cos\phi)(2-p)\theta_1+\frac{1}{3}(\cos\phi-1)(2\cos\phi+1)(p^2-3p+3)\theta_2-\frac{4}{3}\theta_3\,,
$$
which, under the conditions of the present result, is minimized at $(\pi,0)$.  Theorem \ref{thm:main} then gives the stated result.
\end{proof}
Note that the condition $\delta>0$ in Corollary \ref{cor:3} is weaker than (\ref{eq:bxcond}) if $2\theta_3<3\theta_2$.

\subsection{Proof of Theorem \ref{thm:main}}\label{sec:proof}

To prove Theorem \ref{thm:main}, we establish the bounds required by Theorem \ref{thm:bu} using the representation (\ref{eq:sol}), where the integral is taken over the straight line joining $t$ and $1$, so that
\begin{equation}\label{eq:frep}
f_h(x)=(1-t)\int_0^1[t+p(1-t)]^{x-1}\exp\left\{\lambda\left[\mu(t)-\mu(t+p(1-t))\right]\right\}\,dp\,,
\end{equation}
and so
\begin{equation}\label{eq:fbd}
|f_h(x)|\leq|1-t|\int_0^1\exp\left\{\lambda\mbox{Re}\left[\mu(t)-\mu(t+p(1-t))\right]\right\}\,dp\,.
\end{equation}
Now, using the definition of $\mu(t)$, we write
\begin{equation}\label{eq:murep}
\mu(t)-\mu(t+p(1-t))=\sum_{m=1}^\infty(t-1)m\mu_m\int_0^p[t+y(1-t)]^{m-1}\,dy\,.
\end{equation}
Using a Taylor expansion for $[t+y(1-t)]^{m-1}$, we write, for any $k=1,2,\ldots$,
\begin{multline*}
[t+y(1-t)]^{m-1}=\sum_{j=1}^k\frac{[(1-y)(t-1)]^{j-1}}{(j-1)!}\prod_{l=1}^{j-1}(m-l)\\ 
+(t-1)^k\left(\prod_{l=1}^k(m-l)\right)\int_y^1\int_{x_1}^1\cdots\int_{x_{k-1}}^1[t+x_k(1-t)]^{m-k-1}\,dx_k\cdots dx_1\,.
\end{multline*}
Hence, from (\ref{eq:murep}),
\begin{multline}\label{eq:++}
\lambda\left[\mu(t)-\mu(t+p(1-t))\right]=\sum_{j=1}^k\frac{(t-1)^j}{j!}\theta_{j-1}\left(1-(1-p)^j\right)\\
+\sum_{m=k+1}^\infty\left(\prod_{l=0}^k(m-l)\right)\lambda_m(t-1)^{k+1}\int_0^p\int_y^1\int_{x_1}^1\cdots\int_{x_{k-1}}^1[t+x_k(1-t)]^{m-k-1}\,dx_k\cdots dx_1\,dy\,,
\end{multline}
and so
$$
\lambda\mbox{Re}\left[\mu(t)-\mu(t+p(1-t))\right]=\sum_{j=1}^k\frac{\mbox{Re}[(t-1)^j]}{j!}\theta_{j-1}\left(1-(1-p)^j\right)+R_k\,,
$$
where
\begin{multline*}
R_k=\sum_{m=k+1}^\infty\left(\prod_{l=0}^k(m-l)\right)\lambda_m\\
\times
\int_0^p\int_y^1\int_{x_1}^1\cdots\int_{x_{k-1}}^1\mbox{Re}\left[(t-1)^{k+1}[t+x_k(1-t)]^{m-k-1}\right]\,dx_k\cdots dx_1\,dy\,.
\end{multline*}
Using the fact that $t$ is in the unit disc in $\mathbb{C}$, and that $x_k\in[0,1]$, we have 
\begin{equation}\label{eq:+++}
\mbox{Re}\left[(t-1)^{k+1}[t+x_k(1-t)]^{m-k-1}\right]\leq|t-1|^{k+1}\leq(2\mbox{Re}[1-t])^{(k+1)/2}\leq2^k\mbox{Re}[1-t]\,,
\end{equation}
where the second inequality uses the fact that $|t-1|^2\leq2\mbox{Re}[1-t]$ for $t$ in the unit disc, and the final inequality follows since $\mbox{Re}[1-t]\leq2$.

Hence, 
$$
R_k\leq\frac{2^k\mbox{Re}[1-t]}{(k+1)!}\theta_k\left(1-(1-p)^{k+1}\right)\leq\frac{2^k\mbox{Re}[1-t]}{k!}\theta_kp\,,
$$
for $p\in[0,1]$, and so when $t=e^{i\phi}$
\begin{equation}\label{eq:expbd}
\lambda\mbox{Re}\left[\mu(t)-\mu(t+p(1-t))\right]\leq-p(1-\cos\phi)g_k(\phi,p)\leq-\delta_k p(1-\cos\phi)\,.
\end{equation}
Hence, (\ref{eq:fbd}) gives
$$
|f_h(x)|\leq\frac{|1-e^{i\phi}|}{\delta_k(1-\cos\phi)}\,,
$$
and we apply Theorem \ref{thm:bu} to obtain our bound on $M_0^{(K)}$.

Similarly, (\ref{eq:frep}) also gives 
\begin{equation}\label{eq:+}
|f_h(x+1)-f_h(x)|\leq|1-t|^2\int_0^1(1-p)\exp\left\{\lambda\mbox{Re}\left[\mu(t)-\mu(t+p(1-t))\right]\right\}\,dp\,,
\end{equation}
in which we may apply the bound (\ref{eq:expbd}) to obtain the required bound on $M_1^{(K)}$ from Theorem \ref{thm:bu}.

\section{Applications}\label{sec:2}

\subsection{Reliability}

Our first application is to compound Poisson approximation of the two-dimensional consecutive $k$-out-of-$n$:$F$ system, as discussed in Section 3.2 of \cite{bc01}.  This system consists of $n^2$ components, laid out on an $n\times n$ square grid.  For a given $T>0$, each component has failed at time $T$ with probability $q$, independently of the other components in the system.  The entire system fails if there is a $k\times k$ subgrid such that all $k^2$ components have failed at time $T$.  Our interest is in compound Poisson approximation for $W$, which counts the number of the $(n-k+1)^2$ (possibly overlapping) $k\times k$ subgrids for which all components have failed at time $T$.  Letting $\psi=q^{k^2}$, the bound (3.10) of \cite{bc01} (stated here for Kolmogorov, rather than total variation, distance) gives
\begin{multline*}
d_K(\mathcal{L}(W),\mathcal{L}(U)) \\
\leq M_1^{(K)}(n-k+1)^2\psi\left((4k^2+12k-3)\psi+4\sum_{r=1}^{k-1}\sum_{s=1}^{k-1}q^{k^2-rs}+4\sum_{s=1}^{k-2}q^{k^2-ks}\right)\,,
\end{multline*}
where the approximating compound Poisson random variable $U$ is defined by
$$
\lambda_j=\frac{1}{j}\psi\left[4\pi_1(j)+4(n-k-1)\pi_2(j)+(n-k-1)^2\pi_3(j)\right]\,,
$$
for $j=1,\ldots,5$ (and $\lambda_j=0$ for $j\geq6$), where the functions $\pi_i(j)$ for $i=1,2,3$ are defined in terms of point probabilities of binomial random variables:  
$\pi_i(j)=\mathbb{P}(\mbox{Bin}(i+1,q^k)=j-1)$.

As noted by \cite{bc01}, if $q$ and $\lambda$ are small, (\ref{eq:genbd}) will suffice for providing a bound on $M_1^{(K)}$.  In the case of larger $\lambda$, \cite{bc01} considers the use of the bound in (\ref{eq:bx99}), noting that $\theta_0=(n-k+1)^2\psi$ and $\theta_1\leq4q^k\theta_0$, so that (\ref{eq:bxcond}) is satisfied if $q^k<1/8$.  Under this condition, we use (\ref{eq:bx99}) to obtain
\begin{equation}\label{eq:relsteinfactor}
M_1^{(K)}\leq\frac{1}{(n-k+1)^2\psi(1-8q^k)}\,.
\end{equation}
We consider the use of Corollary \ref{cor:3} to provide such a bound in the case where $q^k\geq1/8$, as an alternative to using a different compound Poisson random variable to obtain an approximation of greater accuracy in cases where the bounds (\ref{eq:genbd}) are too crude.  To that end, we note that straightforward calculations using the definitions of the $\lambda_j$ above give us that
\begin{eqnarray*}
\theta_1&=&4[2+3(n-k-1)+(n-k-1)^2]\psi q^k\,,\\
\theta_2&=&4[2+6(n-k-1)+3(n-k-1)^2]\psi q^{2k}\,,\\
\theta_3&=&24[n-k-1+(n-k-1)^2]\psi q^{3k}\,.
\end{eqnarray*}
From this, it is easy to see that $\theta_2<2\theta_1$ for all $n$ and $k$ if $q^k<2/3$, in which case we are in a position to apply Corollary \ref{cor:3}.  Similarly, $2\theta_3<3\theta_2$ in this case, so we expect Corollary \ref{cor:3} to yield a condition on $U$ weaker than that imposed by (\ref{eq:bxcond}).

Now, straightforward calculations show that, in the notation of Corollary \ref{cor:3},
\begin{equation}\label{eq:reliability}
\psi^{-1}\delta=4a(q^k)+4(n-k-1)b(q^k)+(n-k-1)^2c(q^k)\,,
\end{equation}
where the functions $a$, $b$ and $c$ are defined by $a(y)=(1-2y)^2$, $b(y)=(1-2y)^3$ and $c(y)=(1-4y)(1-4y+8y^2)$.  Hence, $\delta>0$ for all $q$ such that $a(q^k)>0$, $b(q^k)>0$ and $c(q^k)>0$. That is, $\delta>0$ if $q^k<1/4$, a weaker condition than that under which (\ref{eq:bx99}) may be applied.  Hence, by Corollary \ref{cor:3}, we have the following: 
\begin{proposition}
For the compound Poisson random variable $U$ defined above, if $q^k<1/4$ then $M_1^{(K)}\leq(2\delta)^{-1}[1+\log^+(\pi\delta)]$, where $\delta$ is given by (\ref{eq:reliability}).
\end{proposition}
This gives a bound whose behaviour, up to logarithmic terms, for large $n$ and small $q$ is similar to that of (\ref{eq:relsteinfactor}), but which is valid under a weaker condition.

\subsection{Runs}

Let $\xi_1,\ldots,\xi_n$ be independent Bernoulli random variables, each with mean $p$.  Let $W=\sum_{i=1}^n\xi_i\xi_{i+1}$ count the number of 2-runs in this sequence, where all indices are treated modulo $n$.  Compound Poisson approximation for $W$ is a well-studied problem (see, for example, \cite{bc01,bx99,d13,pc10} and references therein), so gives us an excellent application within which to examine the benefit of our Theorem \ref{thm:main}.

Following, for example, \cite{bc01}, we approximate $W$ by a compound Poisson random variable $U$ with $\lambda_1=np^2(1-p)^2$, $\lambda_2=np^3(1-p)$, $\lambda_3=(1/3)np^4$ and $\lambda_j=0$ for $j\geq4$.  Straightforward calculations then give $\theta_0=np^2$, $\theta_1=2np^3$, $\theta_2=2np^4$ and $\theta_j=0$ for $j\geq3$.  We note that we thus always have $\theta_2<2\theta_1$, so that Corollary \ref{cor:3} may be applied with
\begin{equation}\label{eq:runs}
\delta=np^2(1-2p)^2\,,
\end{equation}
which is positive provided that $p\not=1/2$.  
\begin{proposition}
For the compound Poisson random variable $U$ defined above, if $p\not=1/2$, then $M_1^{(K)}\leq(2\delta)^{-1}[1+\log^+(\pi\delta)]$, where $\delta$ is given by (\ref{eq:runs}).  
\end{proposition}
For comparison, (\ref{eq:bxcond}) is valid only under the stronger condition that $p<1/4$, in which case (\ref{eq:bx99}) gives $M_1^{(K)}\leq[np^2(1-4p)]^{-1}$.  Up to logarithmic terms, these two bounds are very similar, though ours is valid for a much wider range of values of $p$.  

These bounds may be applied, for example, with the compound Poisson approximation result 
$$
d_{K}(\mathcal{L}(W),\mathcal{L}(U))\leq3M_1^{(K)}np^4\,,
$$  
given in Section 2.2 of \cite{d13} (again, stated here in terms of Kolmogorov, rather than total variation, distance).

Note also that it is easy to show that (\ref{eq:inc}), and the weaker version of this condition derived in \cite{bu98}, hold provided that $p\leq1/3$, which is again a stronger condition than we need to apply Corollary \ref{cor:3}.

Several other compound Poisson approximation results for $W$ are also available in the literature; see Section 3.1 of \cite{bc01} for a discussion.  For example, Theorem 5.2 of \cite{bx99} gives an upper bound of order $O(p/\sqrt{n})$ on the total variation distance between $W$ and a different compound Poisson random variable to that considered here.  This is asymptotically better than the bounds we have discussed here, but note that this comes at the price of somewhat larger constants in the bound, and again holds only in the case that $p<1/4$.  Since their approximating compound Poisson random variable has $\lambda_j=0$ for all $j\geq3$, it is not possible to use our Theorem \ref{thm:main} to extend the range of values of $p$ for which their result applies.  Similar bounds, as well as further asymptotic expansions, are also given by \cite{pc10}.  

\section{Relaxing the Barbour--Xia condition}\label{sec:gen2}

If our compound Poisson random variable $U$ is such that $\lambda_j=0$ for all $j\geq3$, then Theorem \ref{thm:main} can offer no benefit over the condition (\ref{eq:bxcond}).  However, analysis along the same lines as in the proof of Theorem \ref{thm:main} allows us, in Theorem \ref{thm:gen2} below, to establish Stein factors which may be applied when (\ref{eq:bxcond}) is violated.  These Stein factors will, in general, have the exponential dependence on the parameters of $U$ exhibited by the bound (\ref{eq:genbd}), though in certain cases, as we will illustrate below, they can offer a much better bound than (\ref{eq:genbd}).

Throughout this section, we will be interested only in compound Poisson random variables such that $2\theta_1>\theta_0$.  In the case where the reverse inequality is true, Barbour and Xia \cite{bx99} have already established Stein factors with good dependence on the $\lambda_j$; the case $\theta_0=2\theta_1$ is pathological in both their analysis and ours.

We begin with the following lemma.
\begin{lemma}\label{lemma:c}
Let $c>1$, and let $U$ be a compound Poisson random variable with 
$$
\frac{\theta_1}{\theta_0}\in\left(\frac{1}{2},\frac{1}{2}+\frac{\log c}{3\theta_0}\right]\,.
$$
Let $\delta=\frac{2\theta_1-\theta_0}{2c\sqrt{\pi}}$. Then
$$
M_0^{(K)}\leq2\sqrt{\frac{2}{\delta}}\quad\mbox{ and }\quad M_1^{(K)}\leq\frac{1}{2\delta}\left[1+\log^+(\pi\delta)\right]\,.
$$
\end{lemma}
\begin{proof}
We use the same notation as in the proof of Theorem \ref{thm:main}.  From (\ref{eq:++}) with the choice $k=1$, we have
$$
\lambda\left[\mu(t)-\mu(t+p(1-t))\right]=-(1-t)p\theta_0+(1-t)^2\sum_{i=1}^\infty i(i-1)\lambda_i\int_0^p\int_y^1\left[t+u(1-t)\right]^{i-2}\,du\,dy\,.
$$  
Now, using (\ref{eq:+++}), we have that $\mbox{Re}\left[(1-t)^2[t+u(1-t)]^{i-2}\right]\leq2\mbox{Re}\left[1-t\right]$ for $t$ in the unit disc in $\mathbb{C}$ and $u\in[0,1]$, and so
\begin{eqnarray*}
\lambda\mbox{Re}\left[\mu(t)-\mu(t+p(1-t))\right]&\leq&-\mbox{Re}[1-t]p\theta_0+2\mbox{Re}[1-t]\theta_1\int_0^p\int_y^1\,du\,dy\\
&=&-\alpha_tp[1-2\theta(1-p/2)]\,,
\end{eqnarray*}
where $\theta=\frac{\theta_1}{\theta_0}$ and $\alpha_t=\theta_0\mbox{Re}[1-t]$.  Using this bound in (\ref{eq:fbd}), we get
\begin{eqnarray*}
|f_h(x)|&\leq&|1-t|\int_0^1\exp\left\{-\alpha_tp\left[1-2\theta\left(1-\frac{p}{2}\right)\right]\right\}\,dp\\
&\leq&|1-t|\exp\left\{\frac{(2\theta-1)^2\alpha_t}{4\theta}\right\}\int_{-\infty}^\infty\exp\left\{-\alpha_t\theta\left(p-\frac{2\theta-1}{2\theta}\right)^2\right\}\,dp\\
&=&|1-t|\sqrt{\frac{\pi}{\alpha_t\theta}}\exp\left\{\frac{(2\theta-1)^2\alpha_t}{4\theta}\right\}\,.
\end{eqnarray*}

We note that, for any $c,y\in\mathbb{R}$, if $y\leq(2/3)\log c$, then $e^y\leq c/\sqrt{y}$.  Applying this with the constant $c$ as in the statement of the lemma, and $y=(4\theta)^{-1}(2\theta-1)^2\alpha_t$, we get
$$
|f_h(x)|\leq\frac{2c\sqrt{\pi}|1-t|}{\alpha_t(2\theta-1)}\,,
$$
if $y\leq d$, where $d=(2/3)\log c$.  This allows us to bound $M_0^{(K)}$ using Theorem \ref{thm:bu}, once we have checked that
\begin{equation}\label{eq:polycond}
\frac{(2\theta-1)^2\alpha_t}{4\theta}\leq d\,,
\end{equation}
for the compound Poisson random variable $U$ defined in the statement of the lemma, and for all $t$ in the unit disc in $\mathbb{C}$.  To that end, note that $\alpha_t\leq2\theta_0$, and so (\ref{eq:polycond}) holds if $\theta_0(2\theta-1)^2\leq2d\theta$, which holds if and only if $\theta\in[\theta_L,\theta_U]$, where
$$
\theta_L=\frac{1}{2}-\left(\frac{\sqrt{d(d+4\theta_0)}-d}{4\theta_0}\right)<\frac{1}{2}\,,
$$
and
$$
\theta_U=\frac{1}{2}+\left(\frac{\sqrt{d(d+4\theta_0)}+d}{4\theta_0}\right)>\frac{1}{2}+\frac{d}{2\theta_0}=\frac{1}{2}+\frac{\log c}{3\theta_0}\,.
$$
Since
$$
[\theta_L,\theta_U]\supseteq\left(\frac{1}{2},\frac{1}{2}+\frac{\log c}{3\theta_0}\right]\,,
$$
the bound on $M_0^{(K)}$ follows.  

A similar argument gives a bound for $M_1^{(K)}$: in place of (\ref{eq:fbd}), we use (\ref{eq:+}) in the above to get
$$
|\Delta f_h(x)|\leq|1-t|^2\sqrt{\frac{\pi}{\alpha_t\theta}}\exp\left\{\frac{(2\theta-1)^2\alpha_t}{4\theta}\right\}\leq\frac{2c\sqrt{\pi}|1-t|^2}{\alpha_t(2\theta-1)}\,,
$$ 
for $\theta\in[\theta_L,\theta_U]$, as above.  We again apply Theorem \ref{thm:bu} to yield a bound on $M_1^{(K)}$.
\end{proof}
In Lemma \ref{lemma:c}, we stated our bound for the values of $\theta$ given, rather than for all $\theta\in[\theta_L,\theta_U]$, since $\theta<1/2$ is already taken care of by \cite{bx99}, and $\theta=1/2$ gives $\delta=0$, and so non-informative bounds in our Stein factors.

Choosing $c=\exp\left\{\frac{3}{2}(2\theta_1-\theta_0)\right\}>1$ in Lemma \ref{lemma:c}, we obtain our main result of this section:
\begin{theorem}\label{thm:gen2}
Let $U$ be a compound Poisson random variable with $2\theta_1>\theta_0$ and let
$$
\delta=\frac{2\theta_1-\theta_0}{2\sqrt{\pi}\exp\left\{\frac{3}{2}(2\theta_1-\theta_0)\right\}}\,.
$$
Then
$$
M_0^{(K)}\leq2\sqrt{\frac{2}{\delta}}\quad\mbox{ and }\quad M_1^{(K)}\leq\frac{1}{2\delta}\left[1+\log^+(\pi\delta)\right]\,.
$$
\end{theorem}
The bounds of Theorem \ref{thm:gen2} are, of course, exponential in the $\lambda_j$.  The advantage of these bounds over (\ref{eq:genbd}) can be seen by considering a compound Poisson random variable with $\lambda_2=\frac{1}{2}\lambda_1+\gamma$, for some moderate $\gamma>0$, and $\lambda_j=0$ for $j\geq3$.  In this case, the bound (\ref{eq:genbd}) is exponential in both $\lambda_1$ and $\gamma$, while our Theorem \ref{thm:gen2} gives bounds which are exponential in $\gamma$, but do not depend on $\lambda_1$.  This may be advantageous if $\lambda_1$ is large.  Some illustrations of this are given below, where we consider compound Poisson approximation for a mixed Poisson distribution, and for a sum of independent random variables.

\subsection{Mixed Poisson distributions}

We illustrate Theorem \ref{thm:gen2} by considering compound Poisson approximation for a mixed Poisson random variable $W\sim\mbox{Po}(\xi)$, where $\xi$ is a positive random variable with mean $\nu$ and variance $\sigma^2$.  Letting $U$ have a compound Poisson distribution with $\lambda_1=\nu-\sigma^2$, $\lambda_2=\sigma^2/2$ and $\lambda_j=0$ for $j\geq3$, the proof of Theorem 6 of \cite{d11} gives
$$
d_K(\mathcal{L}(W),\mathcal{L}(U))\leq1.2M_1^{(K)}\mathbb{E}|\xi-\nu|^3\,,
$$
assuming that $\nu>\sigma^2$.  If we have $\nu>2\sigma^2$, from (\ref{eq:bx99}) we have the bound $M_1^{(K)}\leq(\nu-2\sigma^2)^{-1}$.  If $\sigma^2<\nu<2\sigma^2$, we cannot employ the results of \cite{bx99}, but we may use our Theorem \ref{thm:gen2}.  If we write $2\sigma^2=\nu+\gamma$, we have the bound $M_1^{(K)}\leq(2\delta)^{-1}[1+\log^+(\pi\delta)]$, where
$$
\delta=\frac{\gamma}{2\sqrt{\pi}\exp\left\{\frac{3\gamma}{2}\right\}}\,,
$$
which gives a reasonable bound as long as $\gamma$ is neither too small nor too large.  By contrast, the general bound (\ref{eq:genbd}) is exponential in $\nu$, and so may be very much worse in the setting with large $\nu$ and moderate $\gamma$.  

\subsection{Independent summands}

Let $W=Z_1+\cdots+Z_n$, where $Z_1,\ldots,Z_n$ are independent integer-valued random variables.  Corollary 4.4 of \cite{bx99} presents a bound in the approximation of $W$ by a compound Poisson random variable with $\lambda_1=2\mathbb{E}W-\mbox{Var}(W)$, $\lambda_2=(1/2)(\mbox{Var}(W)-\mathbb{E}W)$ and $\lambda_j=0$ for $j\geq3$.  Several other related bounds are also presented; we focus on this only for concreteness. Given that (\ref{eq:bxcond}) is satisfied if and only if $\mathbb{E}W>(2/3)\mbox{Var}(W)$, the bound presented by \cite{bx99} applies if $(2/3)\mbox{Var}(W)<\mathbb{E}W<2\mbox{Var}(W)$.  If we have instead that $(1/2)\mbox{Var}(W)<\mathbb{E}W<(2/3)\mbox{Var}(W)$, we may not apply the results of \cite{bx99} directly, but we may replace their use of the bound (\ref{eq:bx99}) with the bound given by our Theorem \ref{thm:gen2} to derive an approximation theorem in the Kolmogorov distance.  In considering cases in which this bound would be not too large, remarks similar to those made above apply.

\vspace{12pt}

\noindent{\bf Acknowledgements} The author thanks Sergey Utev for invaluable initial discussions related to this work.

\end{document}